\documentclass[12pt]{amsart}
\usepackage[alphabetic]{amsrefs}
\usepackage{mathdots, blkarray}
\usepackage{multicol}
\usepackage{graphicx}
\usepackage{amscd}
\usepackage{upgreek}
\usepackage{stmaryrd}
\usepackage{longtable}
\usepackage{caption}
\usepackage[T1]{fontenc}
\usepackage{latexsym, amsmath, amssymb, amsthm}
\usepackage[svgnames]{xcolor}
\usepackage{rsfso}
\usepackage[mathscr]{eucal}
\usepackage{mathtools}
\usepackage{mathptmx} 
\usepackage{titletoc}
\usepackage{wrapfig}
\usepackage{float}
\usepackage{xypic}
\usepackage{microtype}
\usepackage{dsfont}
\usepackage{xcolor}
\usepackage{color}
\usepackage[colorlinks = true,
            linkcolor  = DarkBlue,
            urlcolor   = DarkRed,
            citecolor  = DarkGreen]{hyperref}
\usepackage{hyperref}
\allowdisplaybreaks	
\usepackage[centering, includeheadfoot, hmargin=1.0in, tmargin=0.5in, 
  bmargin=0.6in, headheight=6pt]{geometry}
  
\newtheorem{theorem}{Theorem}[section]
\newtheorem{prop}[theorem]{Proposition}

\newtheorem{lem}[theorem]{Lemma}
\newtheorem{coro}[theorem]{Corollary}

\newtheorem{thm}[theorem]{Theorem}

\newtheorem{rem}[theorem]{\rm\textsc{Remark}}
\newtheorem{exam}[theorem]{\rm\textsc{Example}}

\newcommand{\ideal}[1]{\ensuremath{\left\langle #1 \right\rangle}}

\newcommand{\oeq}[1]{\ensuremath{\overset{(\ref{#1})}{=}}}
\newcommand{\bslash}{\kern-0.1em\texttt{\scalebox{0.6}[1]{/}}\kern-0.15em \texttt{\scalebox{0.6}[1]{/}}}

\DeclareMathOperator{\Aut}{Aut}

\DeclareMathOperator{\ann}{Ann}
\DeclareMathOperator{\LM}{LM}
\DeclareMathOperator{\LC}{LC}
\DeclareMathOperator{\lcm}{lcm}

\DeclareMathOperator{\Der}{Der}
\DeclareMathOperator{\spol}{spol}
 
\newcommand{\B}{\mathcal{B}}

\newcommand{\D}{\mathcal{D}} 
\newcommand{\C}{\mathbb{C}}

\newcommand{\N}{\mathbb{N}} 
\newcommand{\K}{\mathbb{K}} 
\newcommand{\g}{\mathfrak{g}} 
\newcommand{\p}{\mathfrak{p}} 

\newcommand{\ra}{\longrightarrow}

\newcommand{\hbo}{$\hfill\Diamond$}

\begin{document}
\title{Rota-Baxter operators on $\upomega$-Lie algebras} 
\def\shorttitle{Rota-Baxter operators on $\upomega$-Lie algebras}

\author{Yin Chen}
\address{School of Mathematics and Physics, Key Laboratory of ECOFM of 
Jiangxi Education Institute, Jinggangshan University,
Ji'an 343009, Jiangxi, China \& Department of Finance and Management Science, University of Saskatchewan, Saskatoon, SK, Canada, S7N 5A7}
\email{yin.chen@usask.ca}

\author{Shan Ren}
\address{School of Mathematics and Statistics, Northeast Normal University, Changchun 130024, China}
\email{rens734@nenu.edu.cn}

\author{Jiawen Shan}
\address{School of Mathematics and Systems Science, Shenyang Normal University, Shenyang 110034, China}
\email{ShanJiawen826@outlook.com}

\author{Runxuan Zhang}
\address{Department of Mathematics and Information Technology, Concordia University of Edmonton, Edmonton, AB, Canada, T5B 4E4}
\email{runxuan.zhang@concordia.ab.ca}

\begin{abstract}
This article explores Rota-Baxter operators on finite-dimensional $\upomega$-Lie algebras over a field of characteristic not 2. We provide several methods for constructing left-symmetric algebras, $\upomega$-Lie algebras, and Hom-Lie algebras via compatible Rota-Baxter operators on a given $\upomega$-Lie algebra. We also study the geometric structures of compatible Rota-Baxter operators of weight $0$ and isometric Rota-Baxter operators of weight $1$ over the field of complex numbers. In particular, we prove that  the affine variety of all isometric Rota-Baxter operators of weight $1$ on any finite-dimensional non-Lie complex simple $\upomega$-Lie algebra  is $1$-dimensional. Furthermore, we show that for every $4$-dimensional non-Lie complex $\upomega$-Lie algebra, there always exists a nilpotent compatible Rota-Baxter operator of weight $0$ such that the induced Hom-Lie algebra is nonabelian but solvable.  
\end{abstract}

\date{\today}
\maketitle \baselineskip=16pt

\dottedcontents{section}[1.16cm]{}{1.8em}{5pt}
\dottedcontents{subsection}[2.00cm]{}{2.7em}{5pt}

\section{Introduction}
\setcounter{equation}{0}
\renewcommand{\theequation}
{1.\arabic{equation}}
\setcounter{theorem}{0}
\renewcommand{\thetheorem}
{1.\arabic{theorem}}

\noindent The theory of Rota–Baxter operators on an algebraic structure was initiated by \cite{Bax60} to explore an analytic formula in probability theory, and later developed by \cite{Rot69} with motivations from combinatorics.
Using Rota–Baxter operators to construct new algebraic structures has substantial ramifications in the study of nonassociative algebras and mathematical physics, whereas describing the structure of a generic Rota-Baxter operator occupies a central place in connecting computational ideal theory, nonassociative algebras, and algebraic geometry; see \cite{AB08,BBGN13, Das20, Uch08, GGHZ25}, and \cite{Guo12}. 
The approach due to Golubchik and Sokolov appeared in \cite[Section 3.1]{GS00}, providing a new way  in constructing left-symmetric algebras via Rota–Baxter operators on Lie algebras, serves as one of our motivational examples. The primary objective in this article is to establish connections between Rota–Baxter operators on $\upomega$-Lie algebras and other nonassociative algebras, and to characterize the geometric structures of  Rota–Baxter operators on finite-dimensional complex $\upomega$-Lie algebras.

Introduced in \cite{Nur07} as a natural generalization of Lie algebras, the concept of $\upomega$-Lie algebras was motivated by the study of symmetric $3$-tensors on real vector spaces with a Riemannian metric, particularly in the setting of isoparametric hypersurfaces in spheres in \cite{BN07}. More precisely, let $L$ be a vector space over a field $\mathbb{K}$ of characteristic $\neq 2$ and $\upomega$ be a skew-symmetric bilinear form on $L$. We say that $L$ is an \textit{$\upomega$-Lie algebra} over $\mathbb{K}$ if there exists a skew-symmetric bilinear bracket $[-,-]:L\times L\ra L$ such that the following \textit{$\upomega$-Jacobi identity} holds:
\begin{equation}\label{Jac}
[[x,y],z]+[[y,z],x]+[[z,x],y]=\upomega(x,y)\cdot z+\upomega(y,z)\cdot x+\upomega(z,x)\cdot y
\end{equation}
for all $x,y,z\in L$. During the past two decades, extensive research has been devoted to the algebraic structures, classifications, and applications of finite-dimensional complex $\upomega$-Lie algebras; see for example, \cite{CZ17, CNY23, Oub24, Oub25}, and \cite{CRSZ26}.

This article studies how Rota–Baxter operators on finite-dimensional $\upomega$-Lie algebras over $\K$ can be used to construct new algebraic structures and explores their geometric properties when $\K=\C$ is the field of complex numbers. Before presenting our main results on such constructions, let us first recall some related concepts about Rota–Baxter operators.  Suppose that $(L,[-,-],\upomega)$ denotes a finite-dimensional  $\upomega$-Lie algebra over $\K$  and $\uplambda\in \K$. A linear map $R:L\ra L$ is a \textit{Rota-Baxter operator of weight} $\uplambda$ on $L$ provided that
\begin{equation}
\label{rbc}
[R(x),R(y)]=R([R(x),y]+[x,R(y)]+\uplambda([x,y]))
\end{equation}
for all $x,y\in L$.
We write $\B_{\uplambda}(L)$ for the variety of all Rota-Baxter operators of weight $\uplambda$ on $L$ and denote by
$\B(L)$  the subvariety of all Rota-Baxter operators of weight $0$ on $L$. Clearly, the zero map belongs to $\B_{\uplambda}(L)$, thus $\B_{\uplambda}(L)\neq \emptyset$ for all $\uplambda\in \K$.

We say that a Rota-Baxter operator $R$ on $L$ is \textit{compatible} if
\begin{equation}
\label{comp}
\upomega(R(x),y)+\upomega(x,R(y))=0
\end{equation}
for all $x,y\in L$. We have seen in \cite{CZZZ18} and \cite[Section 2]{CRSZ26} that compatible linear maps on $\upomega$-Lie algebras play an important role in understanding the structures and representations of $\upomega$-Lie algebras.  Write $\B_{\uplambda}^c(L)$ for the variety of all compatible Rota-Baxter operators of weight $\uplambda$ on $L$ and $\B^c(L)$ for the subvariety of all compatible Rota-Baxter operators of weight $0$ on $L$. The fact that $0\in \B_{\uplambda}^c(L)$ implies that  $\B_{\uplambda}^c(L)$ is also nonempty for all $\uplambda\in\K$. Moreover, a Rota-Baxter operator $R$ on $L$ is called \textit{isometric} if
\begin{equation}
\label{isom}
\upomega(R(x),R(y))=\upomega(x,y)
\end{equation}
for all $x,y\in L$. Denote by $\B_{\uplambda}^i(L)$  the subvariety of all isometric Rota-Baxter operators of weight $\uplambda$ on $L$,  write $\Aut(L)$ for the automorphism group of $L$, and define 
\begin{equation}
\ker(\upomega):=\{x\in L\mid \upomega(x,y)=0,\textrm{ for all }y\in L\}.
\end{equation}

We now briefly summarize the main results of this article. The first result  in Proposition \ref{prop2.4} gives an $\upomega$-version of Golubchik-Sokolov's construction, providing a way to construct left-symmetric algebras via Rota-Baxter operators of weight $0$ on $\upomega$-Lie algebras. Our second result, stated in Theorem \ref{thm2.7}, presents a new method to construct $\upomega$-Lie algebras by capitalizing on compatible Rota-Baxter operators of a given $\upomega$-Lie algebra. Theorem \ref{thm2.12} is the third result, which offers a way to construct Hom-Lie algebras via nilpotent compatible Rota-Baxter operators of  $\upomega$-Lie algebras. These three results reveal close and unexpected connections between $\upomega$-Lie algebras, left-symmetric algebras, and Hom-Lie algebras, and further motivates us to explicitly compute compatible (of weight $0$) and isometric (of weight $1$) Rota-Baxter operators of low-dimensional complex $\upomega$-Lie algebras. An important conclusion appeared in Corollary \ref{coro3.10} shows that the affine variety $\B^i_1(L)$ of all isometric Rota-Baxter operators of weight $1$ is always 
$1$-dimensional, where $L$ denotes a finite-dimensional non-Lie simple $\upomega$-Lie algebra over $\C$.
More interestingly, as an application of Theorem \ref{thm2.12}, we obtain that if $L$ is a $4$-dimensional non-Lie $\upomega$-Lie algebra over $\C$, then there always exists a nilpotent compatible Rota-Baxter operator of weight $0$ such that the induced Hom-Lie algebra is nonabelian and solvable; see Corollary \ref{coro4.7} for details.

To derive Corollaries \ref{coro3.10} and \ref{coro4.7}, we note that all non-Lie complex $\upomega$-Lie algebras of dimension at most $4$ have been classified by \cite{CLZ14} and \cite{CZ17}. Previously, these low-dimensional $\upomega$-Lie algebras were shown to be indispensable for understanding the structures and representations of $\upomega$-Lie algebras. For instance, finite-dimensional non-Lie simple complex $\upomega$-Lie algebras must be 3-dimensional; see \cite[Theorem 7.1]{CZ17}.  
We also explicitly  compute Rota-Baxter operators of low-dimensional complex $\upomega$-Lie algebras and explore the geometric structures of the corresponding Rota-Baxter operator varieties. In particular, we will determine all irreducible components of these varieties, using classical Gr\"obner bases and techniques from computational ideal theory as our main tools. Similar methods have been applied  recently to the study of various other nonassociative algebras; see for example, \cite{CZ23, CZ24, RZ24, CZ26}, and \cite{CQRZ26}.

We organize this article as follows. Section \ref{sec2} establishes connections between Rota-Baxter operators of $\upomega$-Lie algebras and other nonassociative algebras. Proposition \ref{prop2.4}, Theorem \ref{thm2.7}, and Theorem \ref{thm2.12} are the main results in this section that illustrate how to use compatible Rota-Baxter operators of $\upomega$-Lie algebras to construct new left-symmetric algebras, $\upomega$-Lie algebras, and Hom-Lie algebras, respectively. We also
explore a close relationship between Rota-Baxter operators and invertible derivations in  Proposition \ref{prop2.1} and a link between Rota-Baxter operators and the representation theory of  $\upomega$-Lie algebras in Proposition \ref{prop2.14}. 
In Section \ref{sec3}, we first provide a procedure to compute Rota-Baxter operators of finite-dimensional complex $\upomega$-Lie algebras and use the technique from computational ideal theory to examine the geometric structure of the affine variety $\B^c(L)$ of compatible Rota-Baxter operators of weight $0$ for any $3$-dimensional non-Lie complex $\upomega$-Lie algebra $L$. As a sample of our procedure, we give all details about the computation of the case when $L=L_1$. In particular, our result shows that  $\B^c(L_1)$ is a $3$-dimensional affine variety with two $3$-dimensional irreducible components; see Corollary \ref{coro3.5}. Applying the same method, we complete all computations on $\B^c(L)$ and $\B^i_1(L)$ summarized  in Tables \ref{tab1} and \ref{tab2}, without going into details. Moreover, these computations confirm Corollary \ref{coro3.10}.  Section \ref{sec4} is devoted to computing the affine variety $\B^s(L)$ of compatible Rota-Baxter operators of weight $0$ with the condition $R^2=0$ on all $4$-dimensional complex $\upomega$-Lie algebras. We examine the nilpotency and solvability of Hom-Lie algebras induced by these Rota-Baxter operators in Table \ref{tab3} and derive Corollary \ref{coro4.7} directly from this table.

Throughout this article, we assume that $\K$ denotes an arbitrary field of characteristic $\neq 2$. Some notations are standard:   $\K^\times:=\K\setminus\{0\}$, $\C$ denotes the field of complex numbers,  $\N^+=\{1,2,3,\dots\}$ denotes the set of all positive integers, and $\N=\{0,1,2,\dots\}$ denotes the set of nonnegative integers. We always use $L$ to denote an $\upomega$-Lie algebra and write $A$ for a nonassociative algebra.  All algebras and representations are finite-dimensional over $\K$ or $\C$.

\section{General Rota-Baxter Operators} \label{sec2}
\setcounter{equation}{0}
\renewcommand{\theequation}
{2.\arabic{equation}}
\setcounter{theorem}{0}
\renewcommand{\thetheorem}
{2.\arabic{theorem}}

\noindent This section, consisting of five subsections,  provides some general properties and constructions of Rota-Baxter operators on $\upomega$-Lie algebras, showing their close connections with other nonassociative algebras such as left-symmetric algebras and Hom-Lie algebras. We especially focus on using Rota–Baxter operators on $\upomega$-Lie algebras to construct new algebraic structures. Throughout this section, we work over an arbitrary field $\K$ of characteristic $\neq 2$, unless stated otherwise. 

\subsection{Rota-Baxter operators and invertible derivations}

Exploring the existence of invertible derivations on a nonassociative algebra is a classical topic, which dates back to
Jacobson's theorem in \cite{Jac55}, demonstrating that a Lie algebra over a field of characteristic zero having an invertible derivation must be nilpotent.  

The following result provides a method for finding all automorphic invertible derivations on a given $\upomega$-Lie algebra $L$ via its automorphic Rota-Baxter operators. 

\begin{prop}\label{prop2.1}
Suppose  $L$ is finite-dimensional. Then $\B(L)\cap \Aut(L)\neq \emptyset$ if and only if $\Der(L)\cap \Aut(L)\neq \emptyset$. 
Moreover, if $\B(L)\cap \Aut(L)\neq \emptyset$, then there exists a one-to-one correspondence between 
$\B(L)\cap \Aut(L)$ and $\Der(L)\cap \Aut(L)$.
\end{prop}

\begin{proof}
Given two arbitrary elements $x,y\in A$ and assume that $R\in \B(L)\cap \Aut(L)$. We claim that $R^{-1}\in \Der(L)\cap \Aut(L)$
and thus $\Der(L)\cap \Aut(L)\neq \emptyset$. In fact, it follows from (\ref{rbc}) that $$[R(x),R(y)]=R([R(x),y]+[x,R(y)]).$$ Note that $R^{-1}$ is also an automorphism of $L$, so its action on both sides implies that 
$$R^{-1}([R(x),R(y)])=[R(x),y]+[x,R(y)]$$ Simplifying this, we see that
$$[x,y]=[R(x),y]+[x,R(y)].$$ Hence,  $R^{-1}([x,y])=R^{-1}([R(x),y]+[x,R(y)])=[x,R^{-1}(y)]+[R^{-1}(x),y]$,
which means that $R^{-1}$ is a derivation of $L$. Thus $R^{-1}\in\Der(L)\cap \Aut(L)$ and the claim holds. 

Conversely, suppose that $R\in\Der(L)\cap \Aut(L)$ is an arbitrary element. Then
$R([x,y])=[x,R(y)]+[R(x),y]$ and so $R^{-2}(R([x,y]))=R^{-2}([x,R(y)]+[R(x),y])$, which can be read as
$$[R^{-1}(x),R^{-1}(y)]=R^{-1}([R^{-1}(x),y]+[x,R^{-1}(y)]).$$
Thus $R^{-1}$ is a Rota-Baxter operator on $L$. Clearly, $R^{-1}\in \B(L)\cap \Aut(L)$ and $\B(L)\cap \Aut(L)\neq \emptyset$.
This completes the proof of the first statement. 

To prove the second statement, we define a map $\varphi: \B(L)\cap \Aut(L)\ra \Der(L)\cap \Aut(L)$ by $R\mapsto R^{-1}$. 
The previous argument shows that $\varphi$ is a well-defined bijective map.
\end{proof}

\begin{rem}{\rm
Note that \cite[Lemma 2.1]{CRSZ26} shows that the compatibility of linear maps on $L$ is preserved under addition, scalar multiplication, and the commutator bracket product. Moreover, this compatibility is also preserved for automorphic derivations on $L$ (i.e., under the map $\varphi$ above).
In fact, assume that $R\in \Der(L)\cap \Aut(L)$ is compatible, then 
$$\upomega(R^{-1}(x),y)=\upomega(R^{-1}(x),R(R^{-1}(y)))=-\upomega(R(R^{-1}(x)),R^{-1}(y))=-\upomega(x,R^{-1}(y))$$
for all $x,y\in L$. This means that $R^{-1}$ is also compatible.  Hence, the above map $\varphi$ restricts to a 
one-to-one correspondence between compatible automorphic Rota-Baxter operators and
compatible automorphic derivations (both on $L$).
\hbo}\end{rem}

\begin{rem}{\rm
The statements in Proposition \ref{prop2.1} and their proofs might be generalized to the cases of other nonassociative algebras; see for example, \cite[Proposition 3.1]{LHB07} for the version of left-symmetric algebras.
\hbo}\end{rem}

\subsection{Rota-Baxter operators and left-symmetric  algebras}

Recall that a nonassociative algebra $A$ is called a \textit{left-symmetric algebra} if 
\begin{equation}
\label{pre-Lie}
(xy)z-x(yz)=(yx)z-y(xz)
\end{equation}
for all $x,y,z\in A$. Left-symmetric algebras, containing all associative algebras as a subclass, are a kind of Lie-admissible algebras that have been studied extensively in the past 50 years because of their connections with convex homogeneous cones, affine structures on Lie groups, and Lie theory;  see \cite{Bai20} for an introductory article on the structures of left-symmetric algebras.

It is well-known that Rota-Baxter operators on Lie algebras can be used to construct new left-symmetric algebras; see for example \cite[Proposition 5.2]{GLBJ16} or \cite[Section 3.1]{GS00} for the original reference. Here, we may obtain an $\upomega$-version of Golubchik-Sokolov's construction.

\begin{prop}\label{prop2.4}
Suppose that $R\in \B(L)$ such that the image of $R$ is contained in $\ker(\upomega)$. The following product  defined by
\begin{equation}
\label{eq2.2}
xy:=[R(x),y]
\end{equation} 
gives rise to a left-symmetric algebra structure on the underlying space of $L$.
\end{prop}

\begin{proof}
Given arbitrary elements $x,y,z\in L$, it follows (\ref{eq2.2}) that  $(xy)z=([R(x),y])z=[R([R(x),y]),z].$
Switching $x$ and $y$ observes that $(yx)z=[R([R(y),x]),z].$ Thus
$$(xy)z-(yx)z=[R([R(x),y]-[R(y),x]),z]=[R([R(x),y]+[x,R(y)]),z]\oeq{rbc} [[R(x),R(y)],z].$$
Similarly, $y(xz)=y([R(x),z])=[R(y),[R(x),z]]=[[z,R(x)],R(y)]$ and $x(yz)=[[z,R(y)],R(x)]$. Thus
$$y(xz)-x(yz)=[[z,R(x)],R(y)]+[[R(y),z],R(x)].$$
Moreover,
\begin{eqnarray*}
(xy)z-(yx)z+y(xz)-x(yz)& = & [[R(x),R(y)],z]+[[z,R(x)],R(y)]+[[R(y),z],R(x)] \\
 & \oeq{Jac} & \upomega(R(x),R(y))\cdot z+ \upomega(z,R(x))\cdot R(y)+ \upomega(R(y),z)\cdot R(x)\\
 &=&0
\end{eqnarray*}
where the last equation holds because the image of $R$ is contained in $\ker(\upomega)$. Hence,
(\ref{eq2.2}) produces a new left-symmetric algebra on  $L$.
\end{proof}

The following example shows that our $\upomega$-version of Golubchik-Sokolov's construction obtained above is nontrivial. 

\begin{exam}\label{exam2.5}
{\rm
Consider the $3$-dimensional  $\upomega$-Lie algebra $L_1$ over $\C$ that appeared in \cite[Theorem 2]{CLZ14}, which can be generated by $\{x,y,z\}$ subject to the following relations:
$$[x,y]=y,[x,z]=0,[y,z]=z\textrm{ and }\upomega(x,y)=1, \upomega(x,z)=\upomega(y,z)=0.$$
We will see in Proposition \ref{prop3.2} that $$R=\begin{pmatrix}
    0 & 0 &d   \\
    0& 0 &e   \\
    0& 0 &0   \\
\end{pmatrix}\in \B^c(L_1)$$
for all $d,e\in\C$. Note that $R(x)=d\cdot z, R(y)=e\cdot z$, and $R(z)=0$, thus the image of $R$ is contained in the subspace 
 $\ideal{z}=\ker(\upomega)$. By Proposition \ref{prop2.4}, the new product (\ref{eq2.2}) gives the following 3-dimensional left-symmetric algebra over $\C$ defined by
 $$xy=-d\cdot z,\quad y^2=-e\cdot z,\quad x^2=yx=xz=zx=yz=zy=z^2=0$$
 which actually is isomorphic to ``$H_5$'' in \cite[Proposition 3.4]{Bai09}.
Note that a complete classification of all 3-dimensional complex left-symmetric algebras has been derived by \cite[Section 3.6]{Bai09}. 
\hbo}\end{exam}

\begin{rem}\label{rem2.6}
{\rm
The construction in Proposition \ref{prop2.4} will demonstrate increasing effectiveness as the dimension of the $\upomega$-Lie algebra $L$ grows, because the dimension of $\ker(\upomega)$ could be large. For example, consider the $4$-dimensional non-Lie complex $\omega$-Lie algebra $L_{1,1}$ appeared in \cite[Section 2]{CZ17}, spanned by $\{e,x,y,z\}$ with the following nontrivial generating relations:
$$[e,y]=-e, ~ [x,y]=y, ~ [y,z]=z,~ \upomega(x,y)=1.$$
Clearly, $\ker(\upomega)=\ideal{e,z}$. A direct computation verifies that
$$R=\begin{pmatrix}
   a   & 0&0&-b   \\
   -\frac{ad}{b}   & 0&0&d   \\
      r   & 0&0&s  \\
    \frac{a^2}{b}   & 0&0&-a   \\
\end{pmatrix}$$
is a Rota-Baxter operator on $L_{1,1}$ for all $a,b,d,r,s\in\C$ and $b\in\C^\times$. Since $R(L_{1,1})$ is contained in the subspace spanned by $\{e,z\}$, choosing different parameters $a,b,d,r,s$ and applying Proposition \ref{prop2.4} may produce some new 4-dimensional complex left-symmetric algebras.
\hbo}\end{rem}

\subsection{Rota-Baxter operators and new $\upomega$-Lie algebras}

We also can use Rota-Baxter operators on an $\upomega$-Lie algebra to construct new $\upomega$-Lie algebra structures on the same underlying space (but different $\upomega$). 

Suppose that $R$ denotes a compatible Rota-Baxter operator of weight 0 on an $\upomega$-Lie algebra  $L$. We define a new bracket product $[-,-]_R$ on the underlying space of $L$ by
\begin{equation}
\label{eq2.3}
[x,y]_R:=[R(x),y]+[x,R(y)]
\end{equation}
and define a new skew-symmetric bilinear form $\upomega_R$ by
\begin{equation}
\label{eq2.4}
\upomega_R(x,y):=\upomega(R(x),R(y))
\end{equation}
for all $x,y\in L$. Clearly, the bracket product $[-,-]_R$ is skew-symmetric, because
$$[x,y]_R+[y,x]_R=[R(x),y]+[x,R(y)]+[R(y),x]+[y,R(x)]=0.$$
Furthermore, 

\begin{thm}\label{thm2.7}
The bracket $[-,-]_R$, together with the bilinear form $\upomega_R(-,-)$, defines an $\upomega$-Lie algebra structure on the underlying space of $L$ $($denoted by $L_R$$)$. 
\end{thm}

\begin{proof}
Given three arbitrary elements $x,y,z$ in $L$, it follows from (\ref{eq2.3}) that
\begin{eqnarray*}
[[x,y]_R,z]_R&=&[[R(x),y],z]_R+[[x,R(y)],z]_R \\
 & = & [R([R(x),y]),z]+[[R(x),y],R(z)]+[R([x,R(y)]),z]+[[x,R(y)],R(z)] \\
 &=&([R([R(x),y])+R([x,R(y)]),z])+[[R(x),y],R(z)]+[[x,R(y)],R(z)]\\
 &\oeq{rbc}&[[R(x),R(y)],z]+[[R(x),y],R(z)]+[[x,R(y)],R(z)].
\end{eqnarray*}
Similarly, we see that
\begin{eqnarray*}
~[[y,z]_R,x]_R&=&[[R(y),R(z)],x]+[[R(y),z],R(x)]+[[y,R(z)],R(x)], \\
~[[z,x]_R,y]_R&=&[[R(z),R(x)],y]+[[R(z),x],R(y)]+[[z,R(x)],R(y)].
\end{eqnarray*}
Using the $\upomega$-Jacobi identity (\ref{Jac}), we observe that
\begin{eqnarray*}
&&\hspace{-0.2cm} [[x,y]_R,z]_R+ [[y,z]_R,x]_R+[[z,x]_R,y]_R\\
& = &\hspace{-0.2cm} (\upomega(R(x),y)+\upomega(x,R(y)))\cdot R(z)+(\upomega(R(y),z)+\upomega(y,R(z)))\cdot R(x)+\\
&&\hspace{-0.2cm}(\upomega(R(z),x)+\upomega(z,R(x)))\cdot R(y)+  \upomega(R(x),R(y))\cdot z+ \upomega(R(y),R(z))\cdot x + \upomega(R(z),R(x))\cdot y\\
&\oeq{comp}&\hspace{-0.2cm}\upomega(R(x),R(y))\cdot z+\upomega(R(y),R(z))\cdot x + \upomega(R(z),R(x))\cdot y\\
&=&\hspace{-0.2cm} \upomega_R(x,y)\cdot z+\upomega_R(y,z)\cdot x + \upomega_R(z,x)\cdot y.
\end{eqnarray*} 
Therefore, the product $[-,-]_R$ gives rise to an $\upomega$-Lie algebra structure on $L$ and we may denote by $L_R$ for 
this $\upomega$-Lie algebra.
\end{proof}

The following example illustrates this construction. 

\begin{exam}{\rm
Let us consider the 4-dimensional complex $\upomega$-Lie algebra $\widetilde{A}_{\upalpha}$ spanned by $\{e,x,y,z\}$ that has appeared in Example \cite[Section 4]{CZ17}. This $\upomega$-Lie algebra is defined by the following nonzero generating relations:
$$[e,z]=e,~[x,y]=x,~[x,z]=x+y,~[y,z]=\upalpha x+z,~ \upomega(y,z)=-1.$$
Suppose that $\upalpha=-\frac{1}{4}$. A tedious but direct computation shows  that
$$R=\begin{pmatrix}
-a&0&0&-\frac{a^2}{b}\\
4b&\frac{c}{2}&c&4a-2c\\
0&-\frac{c}{4}&-\frac{c}{2}&c\\
b&0&0&a
\end{pmatrix}$$
with $a,c\in\C$ and $b\in\C^\times$, is a compatible Rota-Baxter operator of weight 0. In particular, we have
\begin{equation*} 
\begin{aligned}
R(e)&=-ae-\frac{a^2}{b}z &\quad R(x)&=4be+\frac{c}{2}x+cy+(4a-2c)z\\
R(y)&=-\frac{c}{4}x-\frac{c}{2}y+cz &\quad R(z) &=be+az.
\end{aligned}
\end{equation*}
Hence, using this action, we see that
\begin{eqnarray*}
[e,x]_R&=&[R(e),x]+[e,R(x)]=\left[-ae-\frac{a^2}{b}z,x\right]+\left[e,4be+\frac{c}{2}x+cy+(4a-2c)z\right]\\
&=& (4a-2c)e+\frac{a^2}{b}(x+y).
\end{eqnarray*}
Similarly, 
\begin{equation*} 
\begin{aligned}
~[e,y]_R & = -\frac{a^2}{4b}x+\frac{a^2}{b}z+ce&\hspace{-1.3cm}[e,z]_R& = 0 &\quad [x,y]_R& = \left(\frac{3c}{2}-a\right)x+cy+(2c-4a)z\\
~[x,z]_R& = 4be+\left(a+\frac{c}{4}\right)x+\left(a+\frac{c}{2}\right)y+cz &~&~& [y,z]_R& = -\frac{2a+c}{8}x-\frac{c}{4}y+\left(a-\frac{c}{2}\right)z.
\end{aligned}
\end{equation*}
Moreover, $\upomega_R(e,x)=\upomega(R(e),R(x))=\upomega\left(-ae-\frac{a^2}{b}z, 4be+\frac{c}{2}x+cy+(4a-2c)z\right)=-\frac{a^2c}{b}$ and
\begin{equation*} 
\begin{aligned}
\upomega_R(e,y)&=\frac{ac}{2b}&\quad \upomega_R(e,z)&=0 &\quad \upomega_R(x,y)&=-2ac\\
\upomega_R(x,z)&=-ac &\upomega_R(y,z)&=\frac{ac}{2}.&
\end{aligned}
\end{equation*}
Apparently, we obtain a non-Lie $\upomega$-Lie algebra when  $a$ and $c$ are nonzero.  
\hbo}\end{exam}

\begin{rem}{\rm
It should be noted that not all compatible Rota-Baxter operators of weight zero yield non-Lie $\upomega$-Lie algebras through the construction described in Theorem \ref{thm2.7}. For example, if we continue to consider the 3-dimensional complex $\upomega$-Lie algebra $L_1$, which previously appeared in Example \ref{exam2.5}, then we will see in Proposition  \ref{prop3.2} later that
$$R=\begin{pmatrix}
   -a   & b& d   \\
    c   & a& e   \\  
    0&0&0
\end{pmatrix}$$
with $a^2+bc=0$ and $ad=be$, is a generic compatible Rota-Baxter operator of weight 0. This means that 
the action of $R$ on $L_1$ can be given by
$$R(x)=-ax+by+dz,~ R(y)=cx+ay+ez,~\textrm{and }R(z)=0.$$
Thus $[x,y]_R=[R(x),y]+[x,R(y)]=[-ax+by+dz,y]+[x,cx+ay+ez]=-ay-dz+ay=-dz$. Similarly, 
$[x,z]_R=bz$ and $[y,z]_R=az$. 
Moreover, $\upomega_R(x,y)=\upomega(R(x),R(y))=\upomega(-ax+by+dz, cx+ay+ez)=-a^2-bc=-(a^2+bc)=0$, because 
 $a^2+bc=0$. As $R(z)=0$, we see that $$\upomega_R(x,z)=\upomega_R(y,z)=0.$$ Hence, the $\upomega$-Lie algebra $(L_1)_R$ is actually a trivial $\upomega$-Lie algebra (i.e.,  a Lie algebra). 
 
 It is well-known that  all 3-dimensional Lie algebras over $\C$ have been classified up to isomorphism, see for example, \cite[Section 1.4]{Jac79}. However, in higher-dimensional cases, 
Theorem \ref{thm2.7} provides a  potential method for constructing new Lie algebras from compatible Rota-Baxter operators on non-Lie $\upomega$-Lie algebras.
\hbo}\end{rem}

\begin{prop}\label{prop2.10}
If $R\in\B^c(L)$, then $R\in\B^c(L_R)$.
\end{prop}

\begin{proof}
We first need to verify (\ref{rbc}) for $\uplambda=0$ and  all $x,y\in L$. It follows from (\ref{eq2.3}) that
$$[R(x),R(y)]_R=[R^2(x),R(y)]+[R(x),R^2(y)].$$
Similarly, 
\begin{eqnarray*}
R([R(x),y]_R)&=&R([R^2(x),y]+[R(x),R(y)])\oeq{rbc}[R^2(x),R(y)]\\
R([x,R(y)]_R)&=&R([R(x),R(y)]+[x,R^2(y)])\oeq{rbc}[R(x),R^2(y)].
\end{eqnarray*}
Hence, $[R(x),R(y)]_R=R([R(x),y]_R)+R([x,R(y)]_R)=R([R(x),y]_R+[x,R(y)]_R)$, for all $x,y\in L$. 
Namely, $R$ is a Rota-Baxter operator of weight 0 on $L_R$. Moreover, since
$$\upomega_R(R(x),y)+\upomega_R(x,R(y))\oeq{eq2.4}\upomega(R^2(x),R(y))+\upomega(R(x),R^2(y))\oeq{comp}0,$$
we see that $R$ is compatible. Therefore, $R\in\B^c(L_R)$.
\end{proof}

\begin{rem}{\rm
By Proposition \ref{prop2.10}, iterative application of the construction given in Theorem 
\ref{thm2.7}, yields further $\upomega$-Lie algebras. More precisely, we define
$$L_0:=(L,[-,-],\upomega)$$
$$L_i:=(L_{i-1},[-,-]_{R^i},\upomega_{R^i}),$$
where $R^i$ denotes the composition of $i$ copies of $R$ for all $i\in\N^+$.
A similar construction for Rota-Baxter operators on left-symmetric algebras can be found in \cite[Section 4]{LHB07}. Moreover, a method of using linear maps  on an algebraic structure and their eigenvectors to construct new Lie algebras and left-symmetric algebras seems to be a recent trend in the structure theory of nonassociative algebras; see for example, \cite{DJ23, DW24, Che25}, and \cite{DM25}.
\hbo}\end{rem}

\subsection{Rota-Baxter operators and Hom-Lie algebras}

Recall that a \textit{Hom-Lie algebra} is a skew-symmetric nonassociative algebra $(\g,[-,-])$ with a linear map $\upalpha:\g\ra\g$ such that
\begin{equation}
\label{ }
[[x,y],\upalpha(z)]+[[y,z],\upalpha(x)]+[[z,x],\upalpha(y)]=0
\end{equation}
for all $x,y,z\in\g$. Hom-algebraic structures have become a popular topic in nonassociative algebras and mathematical physics since the work in \cite{LS05}; see for example, \cite{BD25, CZ23, GL21}, and \cite{GZZ18} for some new developments on this topic. 

\begin{thm}\label{thm2.12}
Suppose that $R\in\B^c(L)$ with $R^2=0$. The skew-symmetric bracket $[-,-]_R$ in $(\ref{eq2.3})$, together with 
$$\upalpha:=R,$$
defines a Hom-Lie algebra structure on $L$.
\end{thm}

\begin{proof}
Given any three elements $x,y,z\in L$, it follows from (\ref{eq2.3}) that
\begin{eqnarray*}
[[x,y]_R,R(z)]_R & = & [[R(x),y]+[x,R(y)],R(z)]_R  \\
& = & [[R(x),y],R(z)]_R+[[x,R(y)],R(z)]_R\\
&=& [R[R(x),y],R(z)]+[[R(x),y],R^2(z)]+[R[x,R(y)],R(z)]+[[x,R(y)],R^2(z)]\\
&=&[R[R(x),y]+R[x,R(y)],R(z)]+[[R(x),y]+[x,R(y)],R^2(z)]\\
&\oeq{rbc}& [[R(x),R(y)],R(z)]+[[R(x),y]+[x,R(y)],R^2(z)]\\
&=&[[R(x),R(y)],R(z)]
\end{eqnarray*}
where the last equality holds because $R^2=0$. Similarly, we have
$$[[y,z]_R,R(x)]_R=[[R(y),R(z)],R(x)]\textrm{ and }[[z,x]_R,R(y)]_R=[[R(z),R(x)],R(y)].$$
Hence, 
\begin{eqnarray*}
&& [[x,y]_R,R(z)]_R+[[y,z]_R,R(x)]_R+[[z,x]_R,R(y)]_R\\
 & = & [[R(x),R(y)],R(z)]+[[R(y),R(z)],R(x)]+ [[R(z),R(x)],R(y)]\\
 & \oeq{Jac} &  \upomega(R(x),R(y))\cdot R(z)+ \upomega(R(y),R(z))\cdot R(x)+ \upomega(R(z),R(x))\cdot R(y)\\
 &\oeq{comp}& -\upomega(x,R^2(y))\cdot R(z)- \upomega(y,R^2(z))\cdot R(x)- \upomega(z,R^2(x))\cdot R(y)\\
 &=&-\upomega(x,0)\cdot R(z)- \upomega(y,0)\cdot R(x)- \upomega(z,0)\cdot R(y)=0.
\end{eqnarray*}
This means that $(L,[-,-]_R,R)$ is a Hom-Lie algebra. 
\end{proof}

\begin{exam}{\rm
Consider the $3$-dimensional complex $\upomega$-Lie algebra $L_2$ appeared in \cite[Theorem 2]{CLZ14}, which is spanned  by $\{x,y,z\}$ subject to the following nonzero relations:
$$[x,z]=y,[y,z]=z,\textrm{ and }\upomega(x,z)=1.$$
We will see in Table \ref{tab1} of Section \ref{sec3} that $$R=\begin{pmatrix}
    0 & a &b   \\
    0& 0 &0   \\
    0& 0 &0   \\
\end{pmatrix}\in \B^c(L_2)$$
for all $a,b\in\C$.  Clearly, $R^2=0$. By Theorem \ref{thm2.12}, the triple $(L_2,[-,-]_R,R)$ forms a Hom-Lie algebra, 
where $[x,y]_R=-bz, [x,z]_R=az$ and $[y,z]_R=0.$ 

Note that two recent works have classified the variety of all $3$-dimensional complex Hom-Lie algebras and explored the associated geometric structures; see  \cite{AV21} and \cite{FCR23}.
\hbo}\end{exam}

\subsection{Rota-Baxter operators and representations of $\upomega$-Lie algebras}

We close this section by establishing a connection between Rota-Baxter operators on an $\upomega$-Lie algebra $L$ and its representations; see \cite[Section 6]{CZZZ18} and \cite{Zha21} for fundamental results on the representation theory of $\upomega$-Lie algebras.

Recall that a finite-dimensional vector space $V$ over $\K$ is called an \textit{$L$-module} if there exists a bilinear map $L\times V\ra V$
given by $(x,v)\mapsto x\cdot v$ such that
\begin{equation}
\label{eq2.6}
[x,y]\cdot v=x\cdot (y\cdot v)-y\cdot (x\cdot v)+\upomega(x,y)\cdot v
\end{equation}
for all $x,y\in L$ and $v\in V$.

Suppose that $V$ is an $L$-module and $R$ is an isometric Rota-Baxter operator of weight 1 on $L$. We write $\ann_L(V)$ for the annihilator of $L$ in $L$, i.e.,
$\ann_L(V):=\{x\in L\mid x\cdot v=0,\textrm{ for all }v\in V\}$ and define a new bilinear map from $L\times V$ to $V$ by
\begin{equation}
\label{eq2.7}
x\ast v:=R(x)v
\end{equation}
for all $x\in L$ and $v\in V$.

\begin{prop}\label{prop2.14}
Suppose that the image of $[R(L),L]$ under $R$ is contained in $\ann_L(V)$. The product $\ast$ in $(\ref{eq2.7})$ 
defines an $L$-module structure on $V$.
\end{prop}

\begin{proof}
It suffices to verify (\ref{eq2.6}) holds for the new product $\ast$. For arbitrary $x,y\in L$ and $v\in V$, note that $x\ast(y\ast v)=x\ast (R(y)\cdot v)=R(x)\cdot (R(y)\cdot v)$ and similarly, $y\ast(x\ast v)=R(y)\cdot (R(x)\cdot v)$, thus
\begin{eqnarray*}
x\ast(y\ast v)-y\ast(x\ast v)+\upomega(x,y)v& = & R(x)\cdot (R(y)\cdot v)-R(y)\cdot (R(x)\cdot v)+\upomega(x,y)\cdot v \\
 & \oeq{eq2.6} & [R(x),R(y)]\cdot v-\upomega(R(x),R(y))\cdot v+\upomega(x,y)\cdot v. 
\end{eqnarray*}
As $R$ is isometric, $\upomega(R(x),R(y))=\upomega(x,y)$ and we see that 
$$x\ast(y\ast v)-y\ast(x\ast v)+\upomega(x,y)v=[R(x),R(y)]\cdot v.$$
Moreover, it follows from (\ref{rbc}) that $[R(x),R(y)]-R([x,y])=R([R(x),y]+[x,R(y)])$. Hence,
\begin{eqnarray*}
x\ast(y\ast v)-y\ast(x\ast v)+\upomega(x,y)v-[x,y]\ast v& = &  [R(x),R(y)]\cdot v- R([x,y])\cdot v\\
&=&([R(x),R(y)]- R([x,y]))\cdot v\\
&=&R([R(x),y]+[x,R(y)])\cdot v
\end{eqnarray*}
which is equal to zero because the image of $[R(L),L]$ under $R$ is contained in $\ann_L(V)$. Therefore,
$$[x,y]\ast v=x\ast(y\ast v)-y\ast(x\ast v)+\upomega(x,y)v$$
and the product $\ast$ gives an $L$-module structure on $V$.
\end{proof}

\section{Rota-Baxter Operators on $3$-dimensional $\upomega$-Lie Algebras}  \label{sec3}
\setcounter{equation}{0}
\renewcommand{\theequation}
{3.\arabic{equation}}
\setcounter{theorem}{0}
\renewcommand{\thetheorem}
{3.\arabic{theorem}}

\noindent In this section, we explicitly compute all compatible Rota-Baxter operators of weight 0 and all isometric Rota-Baxter operators of weight 1 on $3$-dimensional $\upomega$-Lie algebras over $\C$ in terms of computational algebraic geometry.  Note that $3$-dimensional complex $\upomega$-Lie algebras have been classified by
\cite[Theorem 2]{CLZ14} into two families ($A_{\upalpha}$ and $C_{\upalpha}$) and three  exceptional $\upomega$-Lie algebras ($L_1,L_2$, and $B$).

Given an $n$-dimensional non-Lie $\upomega$-Lie algebra $L$ over $\C$, assume that $\{e_1,e_2,\dots,e_n\}$ denotes a basis of $L$. We may take the following procedure to obtain an explicit description on elements in $\B_{\uplambda}(L)$ and analyze the geometric structure of $\B_{\uplambda}(L)$:

\begin{enumerate}
  \item Compute all nonzero generating relations among these $e_i$ and determine the values of $\upomega(e_i,e_j)$ for all $i,j\in\{1,\dots,n\}$;
  \item Consider a generic $n\times n$ complex matrix $R=(x_{ij})$ and define the action of $R$ on $L$ by:
  $$R(e_i):=\sum_{j=1}^n x_{ij}\cdot e_j$$
   We define $V(L):=\{(e_i,e_j)\mid 1\leqslant i,j\leqslant n\}$;
  \item Verify (\ref{rbc}) for all $(e_i,e_j)\in V(L)$ and use the linear independence of $\{e_1,\dots,e_n\}$ to obtain finitely many equations  involving $x_{ij}$.  Write $E$ for the set of all such equations;
  \item Use computational ideal theory to solve the system formed by all equations in $E$. Note that we need to make the number of indeterminates as less as possible;
  \item Choosing some suitable  $x_{ij}$ to eliminate other $x_{ij}$ gives us an explicit description on the  matrix form of $R$;
  \item Use the Gr\"obner basis method in computational ideal theory to find all irreducible components of $\B_{\uplambda}(L)$.
\end{enumerate}

\begin{rem}\label{rem3.1}
{\rm
To calculate  $B^c(L)$, we need to set $\uplambda=0$ and add the additional condition (\ref{comp})  into Step 
(3) of the above procedure. Namely, together with the same Steps (1), (2), (4), (5), replacing Step (3) by
\begin{enumerate}
  \item[(3')] Verify (\ref{rbc}) with $\uplambda=0$ and (\ref{comp}) for all $(e_i,e_j)\in V(L)$ and use the linear independence of $\{e_1,\dots,e_n\}$ to obtain finitely many equations  involving $x_{ij}$.  Write $E$ for the set of all such equations;
\end{enumerate}
obtains a procedure to compute the generic element $R$ in $\B^c(L)$. Similarly, to calculate  $B_1^i(L)$, we need to set $\uplambda=1$ and add the extra condition (\ref{isom})  into the above Step (3).
\hbo}\end{rem}

Throughout this section we  focus on the 3-dimensional case and suppose that 
$$R=\begin{pmatrix}
    x_{11}  & x_{12} &x_{13}   \\
 x_{21}  & x_{22} &x_{23}   \\
 x_{31}  & x_{32} &x_{33}   \\
\end{pmatrix}$$
denotes an arbitrary Rota-Baxter operator of weight $\uplambda$ on a 3-dimensional $\upomega$-Lie algebra $L$ over $\C$.  

We write $A$ for the polynomial algebra $\C[x_{ij}\mid 1\leqslant i,j\leqslant 3]$.
The following Propositions \ref{prop3.2}, \ref{prop3.3}, and Corollary \ref{coro3.5} provide a sample to apply 
the above procedure to compute Rota-Baxter operators of $L$ and describe the geometric structure on $\B^c(L)$.

\begin{prop}\label{prop3.2}
The set $\B^c(L_1)$ of all compatible Rota-Baxter operators of weight $0$ on $L_1$ is an affine variety determined by the ideal
$I:=\ideal{x_{31},x_{32},x_{33},x_{11}+x_{22},x_{12}x_{21}+x_{22}^2,x_{12}x_{23}-x_{13}x_{22}}$ in $A$. In particular, if $R\in \B^c(L_1)$, then
\begin{equation}
\label{eq3.1}
R=\begin{pmatrix}
    -a  & b &d   \\
 c  & a &e   \\
0  & 0 &0   \\
\end{pmatrix}
\end{equation}
where $a^2+bc=0$ and $ad=be$.
\end{prop}

\begin{proof}
Note that the nonzero generating relations in $L_1$ are given by
$$[x,y]=y,[y,z]=z,\upomega(x,y)=1.$$
The Rota-Baxter condition (\ref{rbc}) with $\uplambda=0$ and $(x,y)\in V(L_1)$ obtains the following 9 equations:
\begin{equation} \label{eq3.2}
\begin{aligned}
x_{23}x_{31} - x_{23}x_{32} - x_{33}^2 &= 0 \quad & x_{21}x_{32} - x_{22}x_{32} - x_{32}x_{33} &= 0\\
x_{21}x_{31} - x_{22}x_{31} - x_{31}x_{33} &= 0 \quad &x_{13}x_{32} + x_{23}x_{32} &= 0\\
x_{12}x_{31} + x_{21}x_{32} &= 0\quad &x_{12}x_{21} - x_{13}x_{32} + x_{22}^2 &= 0\\\
x_{11}x_{32} - x_{12}x_{31} - x_{12}x_{32} - x_{22}x_{32} &= 0 \quad &x_{11}x_{21} - x_{13}x_{31} + x_{21}x_{22} &= 0\\
x_{11}x_{23} - x_{12}x_{23} + x_{13}x_{22} - x_{13}x_{33} + x_{22}x_{23} &= 0.
\end{aligned}
\end{equation}
The compatibility condition (\ref{comp}) when $(x,y)$ runs over $V(L_1)$  can be read as
\begin{equation}
\label{eq3.3}
x_{31}=x_{32}=0\textrm{ and } x_{11} + x_{22}=0.
\end{equation}
Putting these equations in (\ref{eq3.2}) and (\ref{eq3.3})  together, we see that
\begin{equation*} 
\begin{aligned}
x_{31}=x_{32}=x_{33}&=0\quad\quad  &x_{11} + x_{22}&=0\\
x_{12}x_{21}  + x_{22}^2&=0 \quad\quad  &x_{12}x_{23} - x_{13}x_{22} &=0.
\end{aligned}
\end{equation*}
This implies that $\B^c(L_1)=V(I)$. Namely, $\B^c(L_1)$ can be determined by the ideal $I$. 
As a direct consequence, we obtain the generic matrix form of $R$.  
\end{proof}

\begin{prop}\label{prop3.3}
Let $\p_1:=\ideal{x_{11},x_{12},x_{22},x_{31},x_{32},x_{33}}$ and $$\p_2:=\ideal{x_{31},x_{32},x_{33}, 
x_{11}+x_{22},x_{12}x_{21}+x_{22}^2,x_{12}x_{23}-x_{13}x_{22},  x_{13}x_{21} + x_{22}x_{23}}$$
be two ideals of $A$. Then 
\begin{enumerate}
  \item $\p_1$ and $\p_2$ both are prime ideals. 
  \item $\B^c(L_1)=V(\p_1)\cup V(\p_2)$. 
\end{enumerate}
\end{prop}

\begin{proof}
(1) It suffices to show that the quotient rings $A/\p_1$ and $A/\p_2$ are integral domains. Since
$$A/\p_1=\C[x_{ij}\mid 1\leqslant i,j\leqslant 3]/\ideal{x_{11},x_{12},x_{22},x_{31},x_{32},x_{33}}\cong \C[x_{13},x_{21},x_{23}]$$
is a polynomial algebra, we see that $A/\p_1$ is an integral domain. Moreover, 
\begin{equation}
\label{eq3.4}
A/\p_2\cong\C[x_{12},x_{13},x_{21},x_{22},x_{23}]/\ideal{x_{12}x_{21}+x_{22}^2,x_{12}x_{23}-x_{13}x_{22},  x_{13}x_{21} + x_{22}x_{23}}
\end{equation}
and thus
$$(A/\p_2)\left[\frac{1}{x_{12}}\right]\cong \C[x_{12},x_{13},x_{22}]\left[\frac{1}{x_{12}}\right]$$
is an integral domain. To see that $A/\p_2$ is an integral domain, it suffices to show the \textit{claim}  that the class of $x_{12}$ in $A/\p_2$  is not a zerodivisor. This is, because, if the claim holds, then the quotient ring $A/\p_2$ can be embedded into $(A/\p_2)\left[\frac{1}{x_{12}}\right]$ and we have just seen that the latter is an integral domain. Thus $A/\p_2$ is also an integral domain. 

Let us focus on the proof of the claim. Define 
$$\ell_1:=x_{12}x_{21}+x_{22}^2,~~\ell_2:=x_{12}x_{23}-x_{13}x_{22},~~ \ell_3:=x_{13}x_{21} + x_{22}x_{23}$$
and write $B$ for the polynomial ring $\C[x_{12},x_{13},x_{21},x_{22},x_{23}]$.
By (\ref{eq3.4}), we see that the claim is equivalent to the statement that the colon ideal
$$\ideal{\ell_1,\ell_2,\ell_3}: x_{12}=\ideal{\ell_1,\ell_2,\ell_3}$$
where the two ideals both are in $B$. 
It follows from \cite[Section 1.2.4]{DK15} that
$$\ideal{\ell_1,\ell_2,\ell_3}: x_{12}=\frac{1}{x_{12}}\cdot (\ideal{\ell_1,\ell_2,\ell_3}\cap \ideal{x_{12}}).$$
Now we may use the standard method of finding Gr\"obner bases of the intersection of two ideals in a polynomial ring to complete the proof of the claim.  Introduce a new variable $z$ over $B$ and consider the ideal 
$J:=\ideal{\ell_1,\ell_2,\ell_3}\cdot z+\ideal{x_{12}}\cdot (1-z)$ in $B[z]$. Taking the grevlex ordering on $B[z]$ with
$z>x_{12}>x_{13}>x_{21}>x_{22}>x_{23}$, and using the standard Buchberger’s algorithm in Lemma \ref{lem3.4} below, we derive that
\begin{equation}\label{eq3.5}
\D:=\{x_{12}\cdot\ell_1, z\cdot \ell_3, x_{12}\cdot\ell_3, z\cdot x_{13}x_{22} - x_{12}x_{23}, x_{12}\cdot\ell_2, 
z\cdot x_{22}^2 + x_{12}x_{21}, z\cdot x_{12} - x_{12}\}
\end{equation}
is a Gr\"obner basis for $J$. By \cite[Algorithm 1.2.1]{DK15},  removing all polynomials in $\D$ that involve the additional variable $z$ will obtain a Gr\"obner basis for $\ideal{\ell_1,\ell_2,\ell_3}\cap \ideal{x_{12}}$. Thus
$$\{x_{12}\cdot\ell_1,x_{12}\cdot\ell_2,x_{12}\cdot\ell_3\}$$
is a Gr\"obner basis for $\ideal{\ell_1,\ell_2,\ell_3}\cap \ideal{x_{12}}$, and $\{\ell_1,\ell_2,\ell_3\}$
is a Gr\"obner basis for $\ideal{\ell_1,\ell_2,\ell_3}: x_{12}$. Therefore, $\ideal{\ell_1,\ell_2,\ell_3}: x_{12}=\ideal{\ell_1,\ell_2,\ell_3}$
and the claim holds. 

(2) Note that in Proposition \ref{prop3.2}, we have seen that $\B^c(L_1)=V(I)$. Thus it suffices to show that
$V(I)=V(\p_1)\cup V(\p_2)$. In fact, since the class of each generator of $I$ in $A/\p_1$ is zero, $I$ is contained in $\p_1$. Thus
$V(\p_1)\subseteq V(I)$. Clearly, $I\subseteq \p_2$ and so $V(\p_2)\subseteq V(I)$. Hence, 
$V(\p_1)\cup V(\p_2)\subseteq V(I)$. Conversely, by \cite[Theorem 7, page 192]{CLO15}, we see that
$V(\p_1)\cup V(\p_2)=V(\p_1\cdot\p_2)$. Thus, it suffices to show that $\p_1\cdot\p_2\subseteq I$. Note that
\cite[Proposition 6, page 191]{CLO15} implies that $\p_1\cdot\p_2$ can be generated by the products of these generators 
for $\p_1$ and $\p_2$. As the first six generators of $\p_2$ and the last three generators of $\p_1$ belong to $I$, we only need to verify that the classes of the following elements 
$$x_{11}\cdot\ell_3,x_{12}\cdot\ell_3,x_{22}\cdot\ell_3$$
in $A/I$ are zero. Working over modulo $I$, we observe that 
\begin{eqnarray*}
x_{11}\cdot\ell_3&=&x_{11}(x_{13}x_{21} + x_{22}x_{23})=(-x_{22})(x_{13}x_{21} + x_{22}x_{23})\\
&=&-x_{13}x_{21}x_{22}-x_{22}^2x_{23}=-x_{21}x_{12}x_{23}-x_{22}^2x_{23}\\
&=&(-x_{23})(x_{12}x_{21}+x_{22}^2)\equiv0.
\end{eqnarray*}
Similarly, $x_{12}\cdot\ell_3\equiv 0$ and $x_{22}\cdot\ell_3\equiv 0$ as well.
This shows that $\B^c(L_1)=V(I)=V(\p_1)\cup V(\p_2)$. 
\end{proof}

\begin{lem}\label{lem3.4}
The set $\D$ defined in $(\ref{eq3.5})$ is a Gr\"obner basis for $J=\ideal{\ell_1,\ell_2,\ell_3}\cdot z+\ideal{x_{12}}\cdot (1-z)$ in $B[z]$, with  the grevlex ordering on $B[z]$ with $z>x_{12}>x_{13}>x_{21}>x_{22}>x_{23}$.
\end{lem}

\begin{proof}
We write $f_1,\dots,f_7$ for the elements of $\D$ in the order appeared in (\ref{eq3.5}) respectively,  i.e., $f_1=x_{12}\cdot\ell_1, f_2=z\cdot\ell_3, \dots, f_7=zx_{12}-x_{12}$. Note that $J$ is generated by 
$\{z\cdot \ell_1, z\cdot \ell_2, f_2, f_7\}$, and a direct verification shows that
\begin{eqnarray*}
z\cdot \ell_1&=&(zx_{22}^2 + x_{12}x_{21})+x_{21}(zx_{12} - x_{12})=f_6+x_{21}\cdot f_7\\
z\cdot \ell_2&=&-(zx_{13}x_{22} - x_{12}x_{23})+x_{23}(zx_{12} - x_{12})=-f_4+x_{23}\cdot f_7.
\end{eqnarray*}
Thus the ideal $J$ can be generated by $\{f_2,f_4,f_6,f_7\}$, which is contained in the ideal $J'$ generated by  $\D$.
Conversely, modulo $J$, we see that $f_1=x_{12}\cdot \ell_1=(zx_{12}-f_7)\cdot \ell_1\equiv x_{12}\cdot z\cdot \ell_1\equiv 0$, and similarly, $f_3\equiv 0$ and  $f_5\equiv 0$. This means that $f_1,f_3,f_5$ are all in $J$. Hence, 
$J=J'$ can be generated by $\D$.

To show that $\D$ is a Gr\"obner basis for $J$, it suffices to show that the $s$-polynomial of any two elements in $\D$
reduces to zero with respect to $\D$; see for example, \cite[Lemma 1.1.8]{DK15} or \cite[Theorem 3, page 105]{CLO15}.
Recall that the $s$-polynomial of any two polynomials $f$ and $g$ is defined as:
\begin{equation}
\label{eq3.6}
\spol(f,g):=\frac{\LC(g)\cdot t}{\LM(f)}\cdot f- \frac{\LC(f)\cdot t}{\LM(g)}\cdot g
\end{equation}
where $t:=\lcm(\LM(f),\LM(g))$ denotes the least common multiple. Note that
\begin{equation*} 
\begin{aligned}
\LM(f_1)&=x_{12}^2x_{21}, \quad & \LM(f_2) &=zx_{13}x_{21}, \quad  &  \LM(f_3) &=x_{12}x_{13}x_{21}, \quad 
&  \LM(f_4) &=zx_{13}x_{22},\\
\LM(f_5)&=x_{12}x_{13}x_{22}, \quad & \LM(f_6)&=zx_{22}^2, \quad & \LM(f_7)&=zx_{12}.
\end{aligned}
\end{equation*}
By  \cite[Proposition 4, page 106]{CLO15}, if $\LM(f_i)$ and $\LM(f_j)$ are coprime, then 
$\spol(f_i,f_j)$ reduces to zero with respect to $\D$. For instance, $\spol(f_1,f_6)$ and $\spol(f_3,f_6)$ both reduce to zero. 
If $\LM(f_i)$ and $\LM(f_j)$ are not coprime,  we may use \cite[Algorithm 1.1.6]{DK15} to show that 
$\spol(f_i,f_j)$ reduces to zero for $i\neq j\in\{1,2,\dots,7\}$. Let's take $(i,j)=(1,2)$ as an example to illustrate this method.  In fact, it follows from (\ref{eq3.6}) that
$$\spol(f_1,f_2)=zx_{12}x_{13}x_{22}^2 - zx_{12}^2x_{22}x_{23}$$
and $\spol(f_1,f_2)=(x_{12}x_{22})\cdot f_4-(x_{12}x_{22}x_{23})\cdot f_7.$ Thus $\spol(f_1,f_2)$
reduces to $0$ with respect to $\D$. A similar and case-by-case argument can be applied to the other cases.
Hence, $\spol(f_i,f_j)$ reduces to zero for all $i$ and $j$. 
\end{proof}

Apparently, Proposition \ref{prop3.3} leads to the following immediate consequence  that describes the geometric structure of the variety $\B^c(L_1)$.

\begin{coro}\label{coro3.5}
The set $\B^c(L_1)$ is a $3$-dimensional affine variety with two $3$-dimensional irreducible components $V(\p_1)$ and $V(\p_2)$.
\end{coro}

\begin{rem}{\rm
Note that the containment $\B^c(L)\subseteq \B(L)$ might be proper. Namely, there might be non-compatible Rota-Baxter operators of weight zero on $\upomega$-Lie algebras.  For example, a direct computation verifies that
$$R=\begin{pmatrix}
    a  & a &b   \\
 0  & 0 &c   \\
0  & 0 &0   \\
\end{pmatrix}\in \B(L_1)$$
for all $a,b,c\in\C$. Clearly, when $a\neq 0$, this $R$ doesn't belong to $\B^c(L_1)$ because it doesn't have the general matrix form appeared in Proposition \ref{prop3.2}.
\hbo}\end{rem}

Table \ref{tab1} below summarizes our computations on $\B^c(L)$ and its geometric structure for a $3$-dimensional non-Lie complex $\upomega$-Lie algebra $L$, where $\dim(\B^c(L))$ denotes the Krull dimension of $\B^c(L)$ and 
``$\#$ Irred.'' denotes the number of irreducible components of $\B^c(L)$. 

\begin{center}
\renewcommand{\arraystretch}{1}
\begin{longtable}{c|c|c|c|c|c}
$L$  &  $\dim(\B^c(L))$ &  $\#$ Irred. & $V(\p_1)$ & $V(\p_2)$ & $V(\p_3)$ \\ \hline
$L_1$ & 3 & 2 &  $\begin{pmatrix}
    0  & 0 &a   \\
 b  & 0 &c   \\
0  & 0 &0   \\
\end{pmatrix}$ & $\begin{pmatrix}
    -a  & b &d   \\
 c  & a &e   \\
0  & 0 &0   \\
\end{pmatrix}$, \parbox[c]{2cm}{\raggedright $a^2+bc=0$\\ $ad=be$\\ $cd+ae=0$} \\ \hline
$L_2$ & 2 & 3 & $\begin{pmatrix}
    0  & 0 &0   \\
 0  & 0 &0   \\
a  & b &0   \\
\end{pmatrix}$ & $\begin{pmatrix}
    0  & a &0   \\
 0  & b &0   \\
0  & 0 &0   \\
\end{pmatrix}$ & $\begin{pmatrix}
    0  & a &b   \\
 0  & 0 &0   \\
0  & 0 &0   \\
\end{pmatrix}$\\ \hline 
$B$ &2&2& $\begin{pmatrix}
    0  & 0 &0   \\
 0  & 0 &0   \\
a  & b &0   \\
\end{pmatrix}$ & $\begin{pmatrix}
    a  & 0 &0   \\
 0  & 0 &0   \\
0  & b &0   \\
\end{pmatrix}$\\ \hline
$A_{\upalpha}$& 3 & 2&$\begin{pmatrix}
    0  & 0 &0   \\
 0  & 0 &0   \\
a  & b &0   \\
\end{pmatrix}$ & $\begin{pmatrix}
    0  & 0 &0   \\
 b & 0 &0   \\
a  & b &0   \\
\end{pmatrix}$\\ \hline
$C_{\upalpha}$&2&3&$ \begin{pmatrix}
    0  & 0 &0   \\
 0& 0 &0   \\
a  & b &0   \\
\end{pmatrix}$ & $\begin{pmatrix}
    0  & 0 &0   \\
 a & 0 &b   \\
0  & 0 &0   \\
\end{pmatrix}$ & $\begin{pmatrix}
    a  & 0 &0   \\
 0 & 0 &0   \\
0  & 0 &0   \\
\end{pmatrix}$ \\
\caption{$\B^c(L)$ in Dimension 3} \label{tab1}
\end{longtable}
\end{center}
\vspace{-0.8cm}

We may derive the following corollary directly from Table \ref{tab1}.

\begin{coro}
Let $L$ be a $3$-dimensional non-Lie $\upomega$-Lie algebra over $\C$. Then
the variety $\B^c(L)$ of compatible Rota-Baxter operators of weight $0$ is not irreducible. Moreover, $\B^c(L)$
has Krull dimension either $2$ or $3$, with at most $3$ irreducible components. 
\end{coro}

Note that Table \ref{tab1} also demonstrates that there are many non-invertible Rota-Baxter operators of weight $0$ on $\upomega$-Lie algebras. The following example shows that invertible Rota-Baxter operators will appear in the case of weight 1. 

\begin{exam}{\rm
Consider the $3$-dimensional complex $\upomega$-Lie algebra $L_1$ and the affine variety $\B^i_1(L_1)$ of all isometric Rota-Baxter operators of weight $1$ on $L_1$. Applying the same method as in computing $\B^c(L_1)$, we see that $\B^i_1(L_1)$ has three irreducible components $V(\p_1), V(\p_2)$, and $V(\p_3)$, which consist of the 
following matrices 
$$\begin{pmatrix}
    -1  & a &b   \\
 0 & -1 &0   \\
0  & 0 &-1   \\
\end{pmatrix}, \quad \begin{pmatrix}
    -1  & -1 &a   \\
 0 & -1 &b   \\
0  & 0 &-1   \\
\end{pmatrix}, \quad \begin{pmatrix}
    -1  & a & -ab-b  \\
 0 & -1 & b  \\
0  & 0 &0   \\
\end{pmatrix}$$
respectively, for all $a,b\in \C$. Clearly, elements in $V(\p_1)$ and $V(\p_2)$ are invertible. 

Hence, $\B^i_1(L_1)$ is  a $2$-dimensional affine variety, because each component $V(\p_j)$ has Krull dimension $2$. 
Furthermore, the vanishing ideals for the three components can be written as
\begin{eqnarray*}
\p_1&=&\ideal{x_{11}+1,x_{21},x_{22}+1,x_{23},x_{31},x_{32},x_{33}+1}\\
\p_2&=&\ideal{x_{11}+1,x_{12}+1,x_{21},x_{22}+1,x_{31},x_{32},x_{33}+1}\\
\p_3&=&\ideal{x_{11}+1,x_{12}x_{23} + x_{13} + x_{23}, x_{21},x_{22}+1,x_{31},x_{32},x_{33}}.
\end{eqnarray*}
It is not difficult to show that $\p_1,\p_2$, and $\p_3$ all are prime ideals. 
\hbo}\end{exam}

In the following Table \ref{tab2}, we summarize the computations on $\B^i_1(L)$ and their geometric structures for all $3$-dimensional non-Lie $\upomega$-Lie algebras over $\C$. 

\begin{center}
\renewcommand{\arraystretch}{1.2}
\begin{longtable}{c|c|c|c|c|c}
$L$  &  $\dim(\B^i_1(L))$ &  $\#$ Irred. & $V(\p_1)$ & $V(\p_2)$ & $V(\p_3)$ \\  \hline
$L_1$ & 3 & 3 &  $\begin{pmatrix}
    -1  & a &b   \\
 0 & -1 &0   \\
0  & 0 &-1   \\
\end{pmatrix}$ & $\begin{pmatrix}
    -1  & -1 &a   \\
 0 & -1 &b   \\
0  & 0 &-1   \\
\end{pmatrix}$ & $\begin{pmatrix}
    -1  & a & -ab-b  \\
 0 & -1 & b  \\
0  & 0 &0   \\
\end{pmatrix}$  \\   \hline 
$L_2$ & 2 & 2 & $\begin{pmatrix}
    -1  & a &b   \\
 0  & -1 &0   \\
0  & 0 &-1   \\
\end{pmatrix}$ &  \parbox[c]{3cm}{\raggedright  $\begin{pmatrix}
    -1  & a &0   \\
 0  & 0 &0   \\
0  & b &-1   \\
\end{pmatrix},$ \\ \quad where $ab=1$} \\ \hline 
$B$ &1&3& $\begin{pmatrix}
    0  & 0 &0   \\
 -1 & -1 &0   \\
a  & 0 &-1  \\
\end{pmatrix}$ & $\begin{pmatrix}
    0  & 0 &0   \\
 1  & -1 &0   \\
a  & 0 &-1   \\
\end{pmatrix}$ & $\begin{pmatrix}
    -1  & 0 &0   \\
 0  & -1 &0   \\
0  & 0 &-1   \\
\end{pmatrix}$\\ \hline
$A_{\upalpha}$& 1 & 2&$\begin{pmatrix}
    -1  & 0 &0   \\
 0  & -1 &0   \\
0  & 0 &-1   \\
\end{pmatrix}$ & $\begin{pmatrix}
    0  & 0 &0   \\
 a & -1 &0   \\
\frac{a^2+a-\upalpha}{2}  & 0 &-1   \\
\end{pmatrix}$\\ \hline
$C_{\upalpha}$&1&2&$ \begin{pmatrix}
    -1  & 0 &0   \\
 0& -1 &0   \\
0  & 0 &-1   \\
\end{pmatrix}$ &  \parbox[c]{3cm}{\raggedright  $\begin{pmatrix}
    0  & 0 &0   \\
 a  & -1 &0   \\
b  & 0&-1   \\
\end{pmatrix},$ \\ where $ab=\frac{1}{\upalpha-1}$}  \\
\caption{$\B^i_1(L)$ in Dimension 3} \label{tab2}
\end{longtable}
\end{center}
\vspace{-0.8cm}

We close this section with two consequences that can be obtained immediately  from Table \ref{tab2}.

\begin{coro}
Let $L$ be a $3$-dimensional non-Lie $\upomega$-Lie algebra over $\C$. Then
the variety $\B^i_1(L)$ of isometric Rota-Baxter operators of weight $1$ is not irreducible. Moreover, $\B^i_1(L)$
has Krull dimension at most $3$, with either $2$ or $3$ irreducible components. 
\end{coro}

Recall that any finite-dimensional non-Lie complex simple $\upomega$-Lie algebras must be one of $\{A_{\upalpha}, B, C_{\upalpha}\}$; see \cite[Theorem 1.7]{CZ17}. Therefore,

\begin{coro}\label{coro3.10}
Let $L$ be a finite-dimensional non-Lie simple $\upomega$-Lie algebra over $\C$. Then
the affine variety $\B^i_1(L)$ is $1$-dimensional.  
\end{coro}

\section{Rota-Baxter Operators and Hom-Lie Algebras}  \label{sec4}
\setcounter{equation}{0}
\renewcommand{\theequation}
{4.\arabic{equation}}
\setcounter{theorem}{0}
\renewcommand{\thetheorem}
{4.\arabic{theorem}}

\noindent Motivated by Theorem \ref{thm2.12}, this section computes all compatible Rota-Baxter operators $R$ of weight 0 with $R^2=0$ on $4$-dimensional non-Lie complex $\upomega$-Lie algebras and as an application, we also study the induced $4$-dimensional  Hom-Lie algebra structures. Note that our calculations in this section will be based
on the classification of $4$-dimensional non-Lie complex $\upomega$-Lie algebras appeared in \cite[Corollary 1.6]{CZ17}.
Throughout  we define 
$$\B^{s}(L):=\{R\in\B^c(L)\mid R^2=0\}.$$

We first give an example of explicit computations on $\B^{s}(L)$ that is an irreducible affine variety and explore the Hom-Lie algebra structure induced by $\B^{s}(L)$. Recall the non-Lie $4$-dimensional complex $\upomega$-Lie algebra  $L_{1,2}$ that can be spanned by $\{e,x,y,z\}$ with the nonzero generating relations:
\begin{equation}
\label{eq4.1}
[e,x]=z,~[e,y]=-e,~ [x,y]=y,~  [y,z]=z,~ \upomega(x,y)=1.
\end{equation}
Performing the procedure appeared in the beginning of Section \ref{sec3} obtains that
$\B^c(L_{1,2})$ has three irreducible components $V(\p_1), V(\p_2)$, and $V(\p_3)$ that consist of the following matrices 
$$R_1:=\begin{pmatrix}
      0& 0 &0 &a  \\
     0 &0&0&b\\
     c&d&0&r\\
     0&0&0&0  
\end{pmatrix},~ R_2:=\begin{pmatrix}
      0& 0 &0 &a  \\
     b &0&0&c\\
     d&-a&0&r\\
     0&0&0&0  
\end{pmatrix},~R_3:=\begin{pmatrix}
      0& 0 &0 &a  \\
     b &c&d&r\\
     s&t&-c&u\\
     0&0&0&0  
\end{pmatrix}$$
respectively. Here, the parameters $a,b,c,d,r$ in $R_1$ and $R_2$ are free over $\C$, while those parameters in $R_3$
are subject to the following relations:
\begin{equation} \label{eq4.2}
\begin{aligned}
ab+du+cr&=0\quad\quad  & as+rt-cu&=0&  \quad\quad &dt+c^2=0\\
bt-cs&=0 \quad\quad  &bc+ds &=0.
\end{aligned}
\end{equation}
Clearly, we see that $(R_1)^2\neq 0$ and $(R_2)^2\neq 0$. It follows from (\ref{eq4.2}) that
$$(R_3)^2=0.$$
This shows that

\begin{prop}
The set $\B^s(L_{1,2})$ is an irreducible affine variety consisting of all $R_3$ above, where entries of $R_3$ are subject to
the relations in $(\ref{eq4.2})$. In particular, the vanishing ideal of $\B^s(L_{1,2})$ in $\C[x_{ij}\mid 1\leqslant i,j\leqslant 4]$
can be generated by $\{x_{11},x_{12},x_{13},x_{41},x_{42},x_{43},x_{44}\}$ and
\begin{equation*} 
\left\{\begin{aligned}
x_{14}x_{21} + x_{23}x_{34} - x_{24}x_{33},&\quad\quad  &  x_{14}x_{31} + x_{24}x_{32} + x_{33}x_{34}, & 
 \quad\quad & x_{23}x_{32} + x_{33}^2,\\
x_{21}x_{32} + x_{31}x_{33},&\quad\quad  &   x_{21}x_{33} - x_{23}x_{31}, & 
 \quad\quad & x_{22} + x_{33}
\end{aligned}\right\}.
\end{equation*}
Moreover, these generators  form a Gr\"obner basis for the vanishing ideal of $\B^s(L_{1,2})$, with respect to 
the lexicographic monomial ordering $x_{11}>x_{12}>\cdots>x_{14}>x_{21}>\cdots>x_{44}$.
\end{prop}

\begin{prop}
The affine variety $\B^s(L_{1,2})$ is $5$-dimensional. 
\end{prop}

\begin{proof}
Note that the Krull dimension of $\B^s(L_{1,2})$ equals the dimension of the coordinate ring, which is defined as the dimension of the vanishing ideal of $\B^s(L_{1,2})$ in the sense of \cite[Section 1.2.5]{DK15}. By \cite[Algorithm 1.2.4]{DK15}, we may choose $M=\{x_{11},x_{12},x_{13},x_{14},x_{21},x_{22},x_{23},x_{41},x_{42},x_{43},x_{44}\}$. Hence, $|M|=11$ and 
the dimension we are looking for equals $16-|M|=16-11=5.$
\end{proof}

\begin{prop}\label{prop4.3}
The Hom-Lie algebra induced by $R_3$ via Theorem $\ref{thm2.12}$ is solvable of length $2$.
\end{prop}

\begin{proof}
We write $(\g,[-,-])$ for this Hom-Lie algebra. Together with (\ref{eq4.1}), applying Theorem \ref{thm2.12} 
implies  that
the nonzero bracket products in $\g$ are given by
\begin{equation*}
\begin{aligned}
[e,x] & =  -de+cz\quad\quad  & [e,y] & = ce-(a - t)z &  \quad\quad & [x,y] & = -be-(r +s)z\\
[x,z] & = dz \quad\quad  & [y,z] & = -cz.
\end{aligned}
\end{equation*}
Clearly, $[\g,\g]$ is contained in the subspace $\ideal{e,z}$. However, since $[e,z]=0$, it follows that
$\g^{(2)}=[\g^{(1)},\g^{(1)}]=[[\g,\g],[\g,\g]]\subseteq [\ideal{e,z},\ideal{e,z}]=0$. Therefore, $\g$ is a $4$-dimensional solvable Hom-Lie algebra of length $2$.
\end{proof}

\begin{rem}{\rm
The induced Hom-Lie algebra in Proposition \ref{prop4.3} might be nilpotent for specific parameters. For example, when $b=c=d=0$, the element $z$ is a central element and $[\g,\g]\subseteq\ideal{z}$. Thus,
$[[\g,\g],\g]\subseteq [\ideal{z},\g]=0$ and $\g$ is nilpotent of class $2$ in this case. 
\hbo}\end{rem}

The affine variety $\B^{s}(L)$ might have multiple irreducible components.  Consider the $\upomega$-Lie algebra  $L_{1,8}$ that can be spanned by $\{e,x,y,z\}$ with the nonzero generating relations:
\begin{equation}
\label{eq4.3}
[e,x]=e+y,~[e,y]=-e+z,~ [x,y]=y,~  [y,z]=z,~ \upomega(e,x)=\upomega(x,y)=1.
\end{equation}
Using the procedure in Section \ref{sec3} again, we deduce that
$\B^c(L_{1,8})$ has three irreducible components $V(\p_1), V(\p_2)$, and $V(\p_3)$ that consist of the following matrices 
$$T_1:=\begin{pmatrix}
      b&-b &0 &a \\
     b &-b&0&a\\
     -b&b&0&c\\
     0&0&0&0  
\end{pmatrix},~ T_2:=\begin{pmatrix}
     a-b &-a & -b& -c \\
     a-b &-a&-b&-c\\
     b-a&a&b&c\\
     0&0&0&0  
\end{pmatrix},~T_3:=\begin{pmatrix}
      0& 0 &0 &a  \\
     b&0&c&d  \\
      0&0&0&r  \\
       0&0&0&0  
\end{pmatrix}.$$
respectively, where $ab+cr=0$ in $T_3$. Clearly, $\dim(V(\p_1))=\dim(V(\p_2))=3$ and $\dim(V(\p_3))=4$.
This means that

\begin{prop}
The affine variety $\B^s(L_{1,8})$ is $4$-dimensional and has three irreducible components. 
\end{prop}

We write $\g_j$ for the Hom-Lie algebra induced by the above $T_j$ for $j=1,2,3$. 

\begin{prop}
$(1)$ $\g_1$ is a nilpotent Hom-Lie algebra. $(2)$ $\g_2$ and $\g_3$ are  solvable Hom-Lie algebras but not nilpotent. 
\end{prop}

\begin{proof}
Let us first look at the Hom-Lie algebra $\g_1$. By Theorem \ref{thm2.12}, a direct verification shows that all nonzero bracket 
products are given by
$$[e,y]=(b-a)z \textrm{ and }[x,y]=(b-a)z.$$
Note that $z$ belongs to the center of $\g_1$, thus $\g_1$ is a nilpotent Hom-Lie algebra of class $2$. 

Similarly, we use Theorem \ref{thm2.12}. In $\g_2$, we observe that $[\g_2,\g_2]$ is contained in the subspace $\{z\}$ and the bracket product of any two of $\{e,x,y,z\}$ is nonzero, thus $z$ is not a central element and $\g_2$ is solvable but not nilpotent. For $\g_3$, as
$[\g_3,\g_3]$ is contained in the subspace $\{e,z\}$ and $z$ is not a central element (because $[x,z]=cz$), it follows that 
$\g_3$ is also solvable but not nilpotent.
\end{proof}

We summarize the geometric structures of $\B^s(L)$ and the Hom-Lie algebra structures induced by $\B^s(L)$ in the following Table \ref{tab3}, where $L$ denotes a $4$-dimensional non-Lie complex $\upomega$-Lie algebra and 
$\g_j$ represents the Hom-Lie algebra induced by the $j$-th irreducible component of $\B^s(L)$.

\begin{center}
\renewcommand{\arraystretch}{1.3}
\begin{longtable}{c|c|c|c|c|c|c}
$L$  &  $\dim(\B^s(L))$ &  $\#$ Irred. & $\g_1$  & $\g_2$ & $\g_3$ & $\g_4$    \\  \hline
$L_{1,1}$ & 6 & 1 & solvable &&& \\ \hline
$L_{1,2}$ & 5 & 1 & solvable &&&\\ \hline
$L_{1,3}$ & 4 & 4 & nilpotent & solvable & solvable & non-solvable \\ \hline
$L_{1,4}$ & 4 & 4 & nilpotent  & solvable  &  solvable & non-solvable  \\ \hline
$L_{1,5}$ & 4 & 4  & nilpotent & solvable   &  solvable  & solvable  \\ \hline
$L_{1,6}$ & 4  &4  & nilpotent & solvable  & solvable & non-solvable  \\ \hline
$L_{1,7}$ & 4 & 4 &  nilpotent & solvable  & solvable   & non-solvable  \\ \hline
$L_{1,8}$ & 4 & 3 & nilpotent & solvable &solvable & \\ \hline
$E_{1,\upalpha}~(\upalpha\neq 0,1)$ & 4 & 4 & nilpotent  & solvable  & solvable & solvable \\ \hline
$F_{1,\upalpha}~(\upalpha\neq 0,1)$ & 4 & 2 & solvable & non-solvable & & \\ \hline
$G_{1,\upalpha}$ & 3 & 4 & nilpotent   & non-solvable  &non-solvable  &non-solvable \\ \hline
$H_{1,\upalpha}$ & 3 & 4 & nilpotent & non-solvable  & non-solvable  & non-solvable  \\ \hline
$L_{2,1}$ & 4  & 4 & nilpotent & solvable  & solvable   & non-solvable \\ \hline
$L_{2,2}$ & 3  & 4 & abelian & nilpotent &  solvable     & non-solvable \\ \hline
$L_{2,3}$ & 4 & 4 & nilpotent & solvable  &  solvable & non-solvable   \\ \hline
$L_{2,4}$ & 3 & 4  & abelian & nilpotent & solvable    & non-solvable  \\ \hline
$\widetilde{A}_{\upalpha}$ & 3 & 4 & nilpotent & non-solvable  & non-solvable  &non-solvable   \\ \hline
$\widetilde{B}$ & 3 & 4 & nilpotent &solvable    &  solvable  & solvable  \\ \hline
$\widetilde{C}_{\upalpha} ~(\upalpha\neq 0,-1)$ & 3 & 4 & nilpotent & nilpotent   &  solvable  & solvable   \\ \hline
\caption{$\B^s(L)$ in Dimension 4} \label{tab3}
\end{longtable}
\end{center}
\vspace{-0.8cm}

By Table \ref{tab3}, we obtain the following corollary. 

\begin{coro}\label{coro4.7}
Let $L$ be a $4$-dimensional non-Lie $\upomega$-Lie algebra over $\C$. Then there always exists a   
compatible Rota-Baxter operator $R$ of weight $0$ with $R^2=0$ such that the Hom-Lie algebra $(L,[-,-]_R,R)$ induced by $R$ via Theorem $(\ref{thm2.12})$ is nonabelian  and solvable. 
\end{coro}

\vspace{3mm}
\noindent \textbf{Acknowledgements}.  
This research was partially supported by a research project from the Liaoning Provincial Department of Education (Grant No. LJ232510166002) and NNSF of China (Grant No. 12561003).
The authors would like to thank the anonymous referee for his/her careful reading, constructive comments, and suggestions.


\raggedright
\end{document}